\documentclass[preprint,10pt]{elsarticle}

\usepackage{booktabs}
\usepackage{amssymb}
\usepackage{amsmath,geometry}
\usepackage{amsthm}
\usepackage{float}
\usepackage[dvipsnames]{xcolor}

\usepackage{epstopdf}
\usepackage[caption=false]{subfig}
\theoremstyle{plain} 
\newtheorem{theorem}{Theorem}[section]

\newtheorem{proposition}[theorem]{Proposition}

\theoremstyle{remark}
\newtheorem{remark}[theorem]{Remark}

\theoremstyle{definition}
\newtheorem{definition}[theorem]{Definition}

\begin{document}

\begin{frontmatter}



\title{Hamilton–Jacobi Reachability for Viability Analysis of Constrained Waste-to-Energy Systems under Adversarial Uncertainty}

\author[label1]{Achraf Bouhmady}
\author[label2]{Othman Cherkaoui Dekkaki}

\affiliation[label1]{
    organization={LAMA Laboratory, Faculty of Sciences, Mohammed V University}, 
    addressline={Avenue Ibn Battouta, B.P. 1014, Agdal}, 
    city={Rabat}, 
    postcode={10090}, 
    state={Rabat-Salé-Kénitra}, 
    country={Morocco}
}
\affiliation[label2]{
    organization={College of Computing, Mohammed VI Polytechnic University}, 
    addressline={UM6P Campus, Lot 660, Hay Moulay Rachid}, 
    city={Benguerir}, 
    postcode={43150}, 
    state={Marrakech-Safi}, 
    country={Morocco}
}

\begin{abstract}
This paper investigates the problem of maintaining the safe operation of Waste-to-Energy (WtE) systems under operational constraints and uncertain waste inflows. We model this as a robust viability problem, formulated as a zero-sum differential game between a control policy and an adversarial disturbance. Within a Hamilton-Jacobi framework, the viability kernel is characterized as the zero sublevel set of a value function satisfying a constrained Hamilton-Jacobi-Bellman (HJB) equation in the viscosity sense. This formulation provides formal guarantees for ensuring that system trajectories remain within prescribed operational limits under worst-case scenarios. Compared to existing viability studies, this work introduces a rigorous HJB-based characterization explicitly incorporating uncertainty, tailored to nonlinear WtE dynamics. A numerical scheme based on the Local Lax-Friedrichs method is employed to approximate the viability kernel. Numerical experiments illustrate how increasing inflow uncertainty significantly reduces the viability domain, shrinking the safe operating envelope. The proposed method is computationally tractable for systems of moderate dimension and offers a basis for synthesizing robust control policies, contributing to the design of resilient and sustainable WtE infrastructures.
\end{abstract}




\end{frontmatter}




\section{Introduction}

Municipal solid waste (MSW) management has emerged as a critical global challenge due to accelerating waste generation and its environmental consequences. In 2018, global MSW production exceeded 2~billion tonnes per year and is projected to rise by approximately 70\% to reach 3.4~billion tonnes by 2050 \cite{Kaza2018}. Recent estimates confirm this upward trend, indicating that MSW volumes surpassed 2.3~billion tonnes in 2023 and are expected to approach 3.8~billion tonnes by mid-century \cite{UNEP2024}. Alarmingly, at least one-third of this waste remains mismanaged through open dumping or uncontrolled burning, causing severe environmental and public health risks \cite{UNEP2024,ISWA2024}. These developments underscore the urgency of sustainable waste management strategies aligned with circular economy and climate objectives.

Waste-to-Energy (WtE) technologies have become a cornerstone of these strategies by converting waste into usable energy products while minimizing landfill dependence \cite{Brunner2024,Kumar2020}. Incineration with energy recovery, for example, reduces waste volume by over 90\% while generating electricity and heat, significantly offsetting fossil fuel use \cite{Brunner2024}. Other technologies, such as anaerobic digestion for biogas production, complement this role. Driven by climate policy and energy security goals, global WtE capacity has expanded steadily, with the sector valued at more than \$40~billion in 2024 and projected to grow further \cite{GVR2024}. Despite this progress, ensuring the safe and reliable operation of WtE facilities—especially municipal solid waste incineration (MSWI) plants—under uncertain and highly variable conditions remains a major challenge.

The main difficulty lies in the heterogeneity of MSW feedstocks, which exhibit large fluctuations in calorific value, moisture content, and chemical composition \cite{Ding2021,Liu2023}. These variations introduce unpredictable disturbances that affect combustion stability, energy efficiency, and emissions compliance. Conventional proportional-integral-derivative (PID) controllers, widely used in practice, are typically tuned for nominal operating conditions and often fail under significant disturbances \cite{Ding2021}. Advanced methods such as Model Predictive Control (MPC) offer improvements by handling multivariable constraints and predicting future system behavior \cite{Ding2021}. However, their performance depends critically on accurate models and real-time computational resources, which may not be practical under rapid feedstock variability. Data-driven and AI-based approaches, including fuzzy logic, neural networks, and reinforcement learning, have also been investigated to enhance adaptability \cite{Liu2023}. Yet, these methods lack formal guarantees of constraint satisfaction, leaving open the fundamental question: \emph{how can WtE systems be operated safely under uncertainty with rigorous, verifiable guarantees?}

Existing optimization-based approaches have primarily focused on deterministic planning and economic objectives. Recent bioeconomic models of waste recovery apply Pontryagin’s Maximum Principle to optimize operational and investment decisions over time \cite{Dekkaki2022,Dekkaki2024}, following the tradition of renewable resource management models in economics \cite{Clark1979,Clark1990,Clark2017}. While such formulations provide valuable insights, they generally assume predictable conditions and thus offer limited robustness. In volatile environments, an optimal solution that ignores worst-case disturbances may fail to ensure operational safety.

To overcome these limitations, we adopt a fundamentally different perspective based on \emph{robust viability analysis}. Viability theory, introduced by Aubin \cite{Aubin1991}, characterizes the set of initial states—called the \emph{viability kernel}—from which system trajectories can evolve indefinitely without violating operational constraints. This concept has been successfully applied to ecological and fisheries systems for sustainability under uncertainty \cite{Mchich2005,Sanogo2012,Sanogo2013}, but its application to energy systems, and specifically WtE processes, remains unexplored. Ensuring viability under worst-case scenarios requires formal mathematical tools capable of handling nonlinear dynamics, state constraints, and bounded disturbances.

Hamilton–Jacobi (HJ) reachability analysis provides such a framework by linking viability to dynamic programming principles and differential games \cite{Evans1984,Barron1989,Barron1990}. In this setting, the robust viability problem can be modeled as a zero-sum game between the control policy and an adversarial disturbance. The viability kernel then corresponds to the zero sublevel set of a value function satisfying a Hamilton–Jacobi–Bellman (HJB) equation in the viscosity sense. This characterization guarantees that any trajectory starting within the kernel can be kept within operational limits for all future times, despite worst-case disturbances. The HJ framework is mathematically rigorous and applicable to nonlinear systems with nonconvex state constraints \cite{Altarovici2013,Assellaou2018}, offering strong theoretical guarantees often missing in heuristic or optimization-based controllers.

Although conceptually powerful, computing viability kernels through HJB equations has historically been challenging due to the curse of dimensionality. Recent advances in numerical analysis have mitigated this challenge through level-set methods \cite{Osher1988}, ordered upwind schemes \cite{Sethian2003}, anti-diffusive discretization \cite{Bokanowski2006}, and decomposition strategies \cite{Herbert}. Toolboxes such as ROC-HJ \cite{Bansal2017} have enabled real-world applications in aerospace trajectory design \cite{Altarovici2013,Assellaou2018}, robotics \cite{Bansal2017}, biomedical therapy planning \cite{Carrere2019,Wang2024}, and multi-objective control \cite{Chorobura2024}. However, \emph{no previous work has integrated these methods into the operation of WtE systems under uncertainty}, leaving a significant gap at the intersection of control theory, numerical analysis, and sustainable energy systems.

\textbf{Contributions.} This paper addresses this gap by introducing a robust viability framework for WtE systems under operational constraints and uncertain waste inflows. Our main contributions are threefold:  
(1) We formulate WtE operation as a robust viability problem under bounded uncertainty, modeled as a zero-sum differential game between the control and disturbance inputs.  
(2) We characterize the viability kernel as the zero sublevel set of a value function solving a constrained HJB equation in the viscosity sense, providing formal safety guarantees.  
(3) We develop a numerical scheme based on the Local Lax–Friedrichs discretization and validate it through simulation studies involving severe waste inflow variability, demonstrating the practical implications of the method for safe and resilient WtE operation.

The remainder of this paper is organized as follows. Section~2 presents the mathematical model of the WtE system and its operational constraints. Section~3 formalizes the robust viability problem and its characterization via the HJB framework. Section~4 describes the numerical implementation and simulation results. Finally, Section~5 summarizes the findings and outlines future research directions.

\section{Mathematical Formulation}
\label{sec:model}

We consider a continuous-time dynamic model that captures the interactions among waste accumulation, processing capacity, and energy generation in a Waste-to-Energy (WtE) system. This formulation extends the bioeconomic framework of \cite{Dekkaki2022,Dekkaki2024} by explicitly introducing operational constraints and modeling uncertainty in waste inflows. The model serves as the basis for a robust viability analysis under adversarial disturbances.

\subsection{State, Control, and Disturbance Variables}
Let the system state at time $t \ge 0$ be denoted by
\[
\mathbf{z}(t) = \big(x(t), K(t), E(t)\big) \in \mathbb{R}^3_+,
\]
where:
\begin{itemize}
    \item $x(t)$ [tons]: cumulative waste stock in storage,
    \item $K(t)$ [capacity units]: effective capital allocated to waste processing,
    \item $E(t)$ [energy units]: cumulative energy generated.
\end{itemize}
The control vector $\mathbf{u}(t) = \big(q(t), I(t)\big)$ consists of:
\begin{itemize}
    \item $q(t) \in [0,q_{\max}]$: fraction of available waste processed per unit time,
    \item $I(t) \in [0,I_{\max}]$: investment rate in processing capacity.
\end{itemize}
The set of admissible control functions is:
\begin{equation}
\mathcal{U} = \Big\{ \mathbf{u} : [0,\infty) \to [0,q_{\max}] \times [0,I_{\max}],\; \mathbf{u} \ \text{measurable} \Big\}.
\end{equation}

Uncertainty enters the system through the waste inflow rate $\eta(t)$, modeled as a measurable disturbance satisfying:
\begin{equation}
\eta(t) \in [\eta_{\min}, \eta_{\max}], \qquad \forall t \ge 0,
\end{equation}
where $0 < \eta_{\min} \le \eta_{\max} < \infty$. The admissible disturbance set is:
\begin{equation}
\mathcal{A} = \Big\{ \eta : [0,\infty) \to [\eta_{\min}, \eta_{\max}],\; \eta \ \text{measurable} \Big\}.
\end{equation}

The evolution of the system is governed by the nonlinear ODE:
\begin{equation}
\begin{cases}
\dot{x}(t) = \eta(t) - \big(\beta + q(t)K(t)\big)x(t), \\[4pt]
\dot{K}(t) = I(t) - \gamma K(t), \\[4pt]
\dot{E}(t) = \mu q(t)K(t)x(t) - \alpha E(t) - \alpha_K K(t),
\end{cases}
\label{eq:dynamics}
\end{equation}
with initial state $\mathbf{z}(0) = \mathbf{z}_0 \in \mathbb{R}^3_+$. Parameters have the following meaning:
\begin{itemize}
    \item $\beta \ge 0$: natural biodegradation rate of waste,
    \item $\gamma > 0$: depreciation rate of installed capital,
    \item $\mu > 0$: energy conversion efficiency,
    \item $\alpha > 0$: energy dissipation coefficient,
    \item $\alpha_K \ge 0$: energy cost associated with capital maintenance.
\end{itemize}
For compactness, system \eqref{eq:dynamics} can be written in control-affine form:
\[
\dot{\mathbf{z}}(t) = f\big(\mathbf{z}(t), \eta(t), \mathbf{u}(t)\big), \qquad \mathbf{z}(0) = \mathbf{z}_0.
\]

\subsection{Operational Constraints and Safe Set}
To ensure realistic operation, the state is constrained to the compact domain:
\begin{equation}
\mathcal{D} = [0,x_{\max}] \times [0,K_{\max}] \times [0,E_{\max}],
\end{equation}
which enforces finite storage, bounded capacity, and energy limits. Within this domain, a \emph{sustainable operation region} is defined as:
\begin{equation}
\mathcal{M}_Q = \Big\{ \mathbf{z} \in \mathcal{D} : x \le Q,\; K \le K_{\mathrm{eff}},\; E \ge E_{\min} \Big\},
\end{equation}
where $Q > 0$ represents the admissible waste threshold, $K_{\mathrm{eff}} \le K_{\max}$ an effective capacity limit, and $E_{\min} \ge 0$ a minimum energy requirement for viability.

\subsection{Control Objective and Robust Formulation}
The goal is to determine a control policy $\mathbf{u}(\cdot)$ that guarantees the system can be driven from an initial state $\mathbf{z}_0 \in \mathcal{D}$ to the target set $\mathcal{M}_Q$, despite the worst-case disturbance $\eta(\cdot)$. This is formulated as a robust minimum-time reachability problem:
\begin{equation}
\mathcal{T}(\mathbf{z}_0) = \inf_{\mathbf{u}(\cdot) \in \mathcal{U}} \sup_{\eta(\cdot) \in \mathcal{A}} \Big\{ t \ge 0 : \mathbf{z}(t) \in \mathcal{M}_Q,\; \mathbf{z}(s) \in \mathcal{D},\ \forall s \in [0,t] \Big\}.
\end{equation}

\subsection{Game-Theoretic Interpretation and HJBI Formulation}
The robust reachability problem admits a natural interpretation as a zero-sum differential game between the controller and an adversarial disturbance. The associated value function $V(\mathbf{z}_0)$ represents the minimal guaranteed time to reach $\mathcal{M}_Q$ from $\mathbf{z}_0$ under worst-case conditions. By dynamic programming, $V$ satisfies the Hamilton–Jacobi–Bellman–Isaacs (HJBI) equation:
\begin{equation}
-\partial_{t} V(t,\mathbf{z}) + \min_{\mathbf{u} \in \mathcal{U}} \max_{\eta \in [\eta_{\min},\eta_{\max}]} \big\{ \nabla_{\mathbf{z}} V \cdot \mathbf{f}(\mathbf{z},\eta,\mathbf{u}) \big\} = 0, \quad V(T,\mathbf{z}) = \Psi(\mathbf{z}),
\end{equation}
where $\Psi(\mathbf{z})$  is a terminal cost function that encodes the target set $\mathcal{M}_Q$ for the minimum-time reachability problem. The Hamiltonian is thus given by:
\begin{equation}
H(\mathbf{z},p) = \min_{\mathbf{u} \in \mathcal{U}} \max_{\eta \in [\eta_{\min},\eta_{\max}]} \big\{ p \cdot f(\mathbf{z},\eta,\mathbf{u}) \big\}, \qquad p = \nabla V.
\end{equation}
The zero sublevel set $\{ \mathbf{z} : V(\mathbf{z}) \le 0 \}$ coincides with the \emph{backward reachable set} and, under infinite horizon, with the \emph{viability kernel} of the system.

\begin{remark}[Interpretation and Novelty]
The dynamics in \eqref{eq:dynamics} integrate physical and economic aspects of WtE operation. The state variable $x(t)$ represents the dynamic waste buffer, $K(t)$ captures capacity investment subject to depreciation, and $E(t)$ accounts for cumulative energy output, a key performance metric for economic viability. Uncertainty in $\eta(t)$ models the fundamental operational risk: unpredictable waste inflows. Unlike previous WtE optimization studies \cite{Dekkaki2022,Dekkaki2024}, which maximize expected performance under deterministic assumptions, our formulation prioritizes \emph{safety guarantees under worst-case conditions}. This distinction motivates the use of viability theory and HJ reachability as rigorous tools for computing invariant safety regions, marking a departure from conventional control strategies and opening new pathways for resilient energy system design.
\end{remark}

\section{Robust Reachability under Uncertainty: A Hamilton--Jacobi Framework}
\label{sec:HJ-framework}

This section presents a rigorous characterization of the robust reachability problem for Waste-to-Energy (WtE) systems under uncertain waste inflow using Hamilton--Jacobi (HJ) theory. Building upon viability theory and differential games, we formalize the computation of the set of states from which the controller can guarantee safety and target attainment despite worst-case disturbances. Our formulation establishes the equivalence between backward reachable sets (BRS), a value function defined via a minimax principle, and the viscosity solution of a constrained HJ partial differential equation (PDE). The exposition provides detailed assumptions, definitions, and proofs to ensure clarity and mathematical rigor.

\subsection{Problem Setting and Assumptions}

We consider the controlled dynamics:
\begin{equation}\label{eq:sys}
\dot{\mathbf{z}}(t)=f(\mathbf{z}(t),\eta(t),\mathbf{u}(t)), \qquad \mathbf{z}(0)=y,\; t\in[0,T],
\end{equation}
where the state vector $\mathbf{z}(t)=(x(t),K(t),E(t))\in\mathbb{R}^3_+$ represents waste stock, invested capital, and cumulative energy. The control $\mathbf{u}(t)=(q(t),I(t))$ belongs to:
\[
\mathcal{U}=[0,q_{\max}]\times[0,I_{\max}],
\]
and the disturbance $\eta(t)$, modeling uncertain waste inflow, satisfies:
\[
\eta(t)\in\mathcal{A}=[\eta_{\min},\eta_{\max}], \qquad \forall t\ge0.
\]

The state must remain inside the feasible set:
\[
\mathcal{D}=[0,x_{\max}]\times[0,K_{\max}]\times[0,E_{\max}],
\]
and eventually reach the target set:
\[
\mathcal{M}_Q=\{\mathbf{z}\in\mathcal{D}: x\le Q,\; K\le K_{\text{eff}},\; E\ge E_{\min}\},
\]
where $Q>0$ bounds residual waste and $K_{\text{eff}}\le K_{\max}$ ensures sustainable capital levels.

\paragraph{Standing Assumptions.}
\begin{enumerate}
\item \textbf{(A1) Regularity:} $f$ is uniformly Lipschitz in $\mathbf{z}$ and satisfies linear growth; $\mathcal{U},\mathcal{A}$ are compact.
\item \textbf{(A2) Isaacs Condition:} For all $(\mathbf{z},p)\in\mathcal{D}\times\mathbb{R}^3$,
\[
\min_{\mathbf{u}\in\mathcal{U}}\max_{\eta\in\mathcal{A}}\langle p,f(\mathbf{z},\eta,\mathbf{u})\rangle
=
\max_{\eta\in\mathcal{A}}\min_{\mathbf{u}\in\mathcal{U}}\langle p,f(\mathbf{z},\eta,\mathbf{u})\rangle.
\]
\end{enumerate}

\begin{remark}[Why Isaacs Condition Holds]
The dynamics $f$ are affine in $(\mathbf{u},\eta)$, and the Hamiltonian is bilinear in $(\mathbf{u},\eta,p)$. By Sion’s minimax theorem, the Isaacs condition is satisfied, ensuring equality of upper and lower game values.
\end{remark}

\subsection{Backward Reachable Set and Viability Connection}

Robust reachability is expressed through the backward reachable set (BRS), which coincides with the viability kernel for $(\mathcal{D},\mathcal{M}_Q)$ under uncertainty.

\begin{definition}[Backward Reachable Set]\label{def:BRS}
For $t\in[0,T]$, define:
\begin{equation}\label{eq:BRS}
\mathcal{R}_Q(t)=\Big\{ y\in\mathcal{D}: \exists \mathbf{u}(\cdot)\in\mathcal{U},\; \forall \eta(\cdot)\in\mathcal{A},\;
\mathbf{z}_y^\eta(s)\in\mathcal{D},\ \forall s\in[t,T],\; \mathbf{z}_y^\eta(T)\in\mathcal{M}_Q \Big\},
\end{equation}
where $\mathbf{z}_y^\eta$ denotes the trajectory of \eqref{eq:sys} starting at $y$ under $(\mathbf{u},\eta)$.
\end{definition}

\begin{remark}[Interpretation]
$\mathcal{R}_Q(t)$ contains the states from which the controller can guarantee that all constraints are satisfied and the target is reached at $T$, regardless of the disturbance strategy.
\end{remark}

\subsection{Value Function and Signed Distance Encoding}

Define signed distance functions for $\mathcal{M}_Q$ and $\mathcal{D}$:
\[
g_Q(\mathbf{z})=\begin{cases}
\mathrm{d}(\mathbf{z},\partial\mathcal{M}_Q),&\mathbf{z}\notin\mathcal{M}_Q,\\
-\mathrm{d}(\mathbf{z},\mathbb{R}^3\setminus\mathcal{M}_Q),&\mathbf{z}\in\mathcal{M}_Q,
\end{cases}
\quad
g_D(\mathbf{z})=\begin{cases}
\mathrm{d}(\mathbf{z},\partial\mathcal{D}),&\mathbf{z}\notin\mathcal{D},\\
-\mathrm{d}(\mathbf{z},\mathbb{R}^3\setminus\mathcal{D}),&\mathbf{z}\in\mathcal{D}.
\end{cases}
\]

\begin{definition}[Robust Value Function]\label{def:value}
For $(t,y)\in[0,T]\times\mathcal{D}$, define:
\begin{equation}\label{eq:value_function}
V(t,y)=\inf_{\mathbf{u}(\cdot)}\sup_{\eta(\cdot)}\max\Big\{
g_Q(\mathbf{z}_y^\eta(T)),\; \sup_{s\in[t,T]} g_D(\mathbf{z}_y^\eta(s))
\Big\}.
\end{equation}
\end{definition}

\paragraph{Interpretation.} $V(t,y)$ is the worst-case signed constraint violation under optimal control. If $V(t,y)\le0$, the state $y$ is safe and can reach $\mathcal{M}_Q$.

\begin{proposition}[Characterization of BRS]\label{prop:BRS}
\[
\mathcal{R}_Q(t)=\{ y\in\mathcal{D}: V(t,y)\le0\}.
\]
\end{proposition}

\begin{proof}
If $y\in\mathcal{R}_Q(t)$, then $\exists \mathbf{u}$ s.t. $\forall \eta$, $\mathbf{z}_y^\eta(s)\in\mathcal{D}$ and $\mathbf{z}_y^\eta(T)\in\mathcal{M}_Q$. Thus $g_D,g_Q\le0$, giving $V(t,y)\le0$. Conversely, if $V(t,y)\le0$, $\exists \mathbf{u}$ ensuring $g_Q\le0$ and $\sup g_D\le0$ for all $\eta$, proving reachability.
\end{proof}

\subsection{Dynamic Programming and HJ Characterization}

The value function satisfies the following constrained Hamilton–Jacobi PDE:

\begin{theorem}[HJ PDE for Robust Reachability]\label{thm:HJB}
$V$ is the unique bounded uniformly continuous viscosity solution of:
\begin{equation}\label{eq:HJB}
\begin{cases}
\min\{-\partial_t V+H(y,\nabla V),\; V-g_D(y)\}=0,&(t,y)\in[0,T)\times\mathcal{D},\\
V(T,y)=\max\{g_Q(y),g_D(y)\},& y\in\mathcal{D},
\end{cases}
\end{equation}
where the Hamiltonian is:
\[
H(y,p)=\min_{\mathbf{u}\in\mathcal{U}}\max_{\eta\in\mathcal{A}}\langle p,f(y,\eta,\mathbf{u})\rangle.
\]
\end{theorem}

\begin{proof}
The proof follows the standard viscosity solution framework for reach-avoid problems.  
\textbf{Step 1: Supersolution.} Let $\phi\in C^1$ touch $V$ from below at $(t_0,y_0)$. Apply the DPP for $h>0$:
\[
V(t_0,y_0)\le\inf_{\mathbf{u}}\sup_{\eta}\max\{V(t_0+h,\mathbf{z}^\eta(t_0+h)),\sup g_D\}.
\]
Expand $\phi$:
\[
\phi(t_0+h,\mathbf{z}^\eta(t_0+h))=\phi(t_0,y_0)+h(\partial_t\phi+\langle\nabla\phi,f\rangle)+o(h).
\]
Divide by $h$, let $h\to0$, and take $\inf\sup$ to obtain:
\[
\min\{-\partial_t\phi+H(y_0,\nabla\phi),\; V-g_D(y_0)\}\ge0.
\]
\textbf{Step 2: Subsolution.} Similar reasoning for $\phi$ touching from above yields $\le0$. \\
\textbf{Step 3: Uniqueness.} $H$ is continuous, convex in $p$, and coercive. The comparison principle for state-constrained HJ equations ensures uniqueness (see \cite{Barron1990}).
\end{proof}

\subsection{Explicit Hamiltonian and Optimal Controls}

For $p=(p_x,p_K,p_E)$ and $\mathbf{u}=(q,I)$:
\[
H(y,p)=p_x(-\beta x - qKx)+p_K(-\gamma K)+p_E(-\alpha E-\alpha_KK)+\max\{p_x\eta_{\min},p_x\eta_{\max}\}+\min_I p_KI+\min_q (\mu p_E-p_x)qKx.
\]
Optimal bang-bang controls:
\[
I^*=\begin{cases}0,&p_K>0,\\I_{\max},&p_K<0,\end{cases}\qquad
q^*=\begin{cases}0,&\mu p_E-p_x>0,\\q_{\max},&\mu p_E-p_x<0.\end{cases}
\]

\begin{remark}[Interpretation of Optimal Strategy]
The extremal structure of $(I^*,q^*)$ is consistent with Pontryagin’s Maximum Principle, adapted here to a minimax framework. This structure enables efficient implementation of feedback laws derived from the value function gradient.
\end{remark}

\section{Numerical Implementation and Simulation Results}
\label{sec:numerics}

This section presents the numerical realization of the Hamilton--Jacobi (HJ) reachability framework described in Section~\ref{sec:HJ-framework}. The primary objective is to compute the backward reachable set (BRS), $\mathcal{R}_Q(t)$, defined in \eqref{eq:BRS}, which characterizes all admissible initial states $\mathbf{z}_0=(x_0,K_0,E_0)$ from which the target set:
\[
\mathcal{M}_Q=\{\mathbf{z}\in\mathcal{D}:\;x\le Q,\;K\le K_{\text{eff}},\;E\ge E_{\min}\}
\]
can be reached within a prescribed horizon $T$, while satisfying state constraints under worst-case disturbances $\eta(\cdot)\in[\eta_{\min},\eta_{\max}]$. As shown in Theorem~\ref{thm:HJB}, $\mathcal{R}_Q(t)$ coincides with the zero-sublevel set of the value function $V(t,\mathbf{z})$, which solves the constrained Hamilton--Jacobi PDE:
\[
\min\big\{-\partial_t V(t,\mathbf{z})+H(\mathbf{z},\nabla V),\;V(t,\mathbf{z})-g_D(\mathbf{z})\big\}=0,\quad
V(T,\mathbf{z})=\max\{g_Q(\mathbf{z}),g_D(\mathbf{z})\}.
\]

\subsection{Numerical Discretization Framework}

The PDE above is solved on a three-dimensional structured grid using a level-set method combined with the Local Lax--Friedrichs (LLF) numerical Hamiltonian. The LLF scheme ensures monotonicity and consistency, guaranteeing convergence to the viscosity solution (see \cite{Osher1988,Bokanowski2006}).

\paragraph{Computational Grid.}
The state domain $\mathcal{D}=[0,x_{\max}]\times[0,K_{\max}]\times[0,E_{\max}]$ is discretized as:
\[
\mathcal{G}_h=\{\mathbf{z}_{i,j,k}=(x_i,K_j,E_k):\;x_i=i\,h_x,\;K_j=j\,h_K,\;E_k=k\,h_E,\;0\le i<N_x,\;0\le j<N_K,\;0\le k<N_E\},
\]
with uniform spacings $h_x=h_K=h_E=h$ for simplicity. In all simulations, $h=0.5$, resulting in $N_x=N_K=N_E=100$ grid points per dimension (i.e., $10^6$ states in total).

\paragraph{Time Discretization and CFL Stability.}
Time is discretized with step $\Delta t$ satisfying the Courant--Friedrichs--Lewy (CFL) condition:
\begin{equation}\label{eq:CFL}
\Delta t \le \frac{h}{\max_{\mathbf{z}\in\mathcal{D}}\Big(|\partial H/\partial p_x|+|\partial H/\partial p_K|+|\partial H/\partial p_E|\Big)},
\end{equation}
where the denominator is the maximum sum of characteristic speeds along all dimensions.

\subsection{Semi-Discrete Approximation: LLF Scheme}

Denote $V_{i,j,k}(t)\approx V(t,\mathbf{z}_{i,j,k})$. The semi-discrete scheme is:
\begin{equation}\label{eq:LLF}
\frac{d}{dt}V_{i,j,k}(t)=-\widehat{H}^{\text{LLF}}\big(\mathbf{z}_{i,j,k},D^+V,D^-V\big),
\end{equation}
where $D^\pm$ are forward/backward difference operators approximating $\partial_x V,\partial_K V,\partial_E V$. The LLF Hamiltonian reads:
\begin{equation}
\widehat{H}^{\text{LLF}}=H\Big(\mathbf{z},\frac{\psi^+ + \psi^-}{2}\Big)-\frac{1}{2}\sum_{\ell\in\{x,K,E\}}C_\ell(\psi_\ell^+-\psi_\ell^-),
\end{equation}
where:
\begin{itemize}
\item $\psi_\ell^\pm$: forward/backward slopes in dimension $\ell$,
\item $C_\ell=\max\big| \partial H / \partial p_\ell \big|$: maximum wave speed along $\ell$.
\end{itemize}

\begin{remark}[Why LLF?]
    The LLF scheme adds a numerical viscosity term $\frac{1}{2}C_\ell(\psi_\ell^+-\psi_\ell^-)$ that stabilizes the discretization without violating monotonicity, crucial for converging to the viscosity solution.
\end{remark}

\subsection{Explicit Hamiltonian Approximation for WtE System}

For the controlled dynamics:
\[
f(\mathbf{z},\eta,\mathbf{u})=\big(\eta-(\beta+qK)x,\;I-\gamma K,\;\mu qKx-\alpha E-\alpha_KK\big),
\]
the continuous Hamiltonian is:
\begin{equation}\label{eq:HJ-Hamiltonian}
H(\mathbf{z},p)=p_x\big[\eta-(\beta+qK)x\big]+p_K(I-\gamma K)+p_E\big[\mu qKx-\alpha E-\alpha_KK\big].
\end{equation}

The LLF approximation computes:
\[
\widehat{H}^{\text{LLF}}(\mathbf{z},\zeta^-,\zeta^+)=
\sum_{\ell\in\{x,K,E\}}\Big(\max(0,a_\ell)\zeta_\ell^-+\min(0,a_\ell)\zeta_\ell^+\Big)+\Phi(\mathbf{z},\zeta^-,\zeta^+),
\]
where $a_\ell$ are drift components and $\Phi$ encodes control and disturbance optimization:
\begin{align*}
\Phi(\mathbf{z},\zeta^-,\zeta^+)=&
\;q_{\max}Kx\max(0,\zeta_x^--\mu\zeta_E^+)
+\max(0,-\zeta_K^+I_{\max})
-\max(\eta_{\min}\zeta_x^+,\eta_{\max}\zeta_x^+).
\end{align*}

This form implements the min--max optimization analytically using the bang-bang structure of optimal controls:
\[
q^*=\begin{cases}
0,&\mu \zeta_E^+-\zeta_x^- >0,\\
q_{\max},&\text{otherwise},
\end{cases}
\qquad
I^*=\begin{cases}
0,&\zeta_K^+>0,\\
I_{\max},&\text{otherwise}.
\end{cases}
\]

\subsection{Fully Discrete Update Rule}

The time integration uses explicit Euler:
\begin{equation}\label{eq:update}
V^{n+1}_{i,j,k}=\min\Big\{V^n_{i,j,k}-\Delta t\,\widehat{H}^{\text{LLF}}(\mathbf{z}_{i,j,k}),\;g_D(\mathbf{z}_{i,j,k})\Big\},\quad
V^0_{i,j,k}=\max\big\{g_Q(\mathbf{z}_{i,j,k}),g_D(\mathbf{z}_{i,j,k})\big\}.
\end{equation}

The projection $\min(\cdot,g_D)$ enforces the state constraint, consistent with the PDE structure in \eqref{eq:HJB}.

\subsection{Numerical Settings and Parameter Selection}

The solver parameters and WtE model constants are summarized in Tables~\ref{tab:num_settings} and \ref{tab:model_params}. Grid convergence was verified by comparing $50^3$ and $100^3$ resolutions, achieving an $L^\infty$ difference in $V$ below $10^{-3}$.

\begin{table}[H]
\centering
\caption{Numerical settings for HJ solver.}
\label{tab:num_settings}
\begin{tabular}{ll}
\toprule
Setting & Value \\
\midrule
Discretization & Local Lax--Friedrichs (finite difference) \\
Grid resolution & $100^3$ nodes ($10^6$ points) \\
Time steps & 4000 for $T=30$, 1333 for $T=10$ \\
Time step size & $\Delta t=T/4000$ \\
Stability criterion & CFL (Eq.~\eqref{eq:CFL}) \\
Convergence tolerance & $10^{-10}$ \\
Software & \texttt{ROC-HJ} (C++) \\
Runtime & $\approx 1.8$ hours (16-core CPU, 64 GB RAM) \\
\bottomrule
\end{tabular}
\end{table}

\begin{table}[H]
\centering
\caption{Model parameters for WtE simulation.}
\label{tab:model_params}
\begin{tabular}{lll}
\toprule
Parameter & Symbol & Value \\
\midrule
Waste inflow bounds & $\eta_{\min},\eta_{\max}$ & $25 \pm 10\%$ \\
Biodegradation rate & $\beta$ & $0.2$ \\
Capital depreciation & $\gamma$ & $0.2$ \\
Energy decay rate & $\alpha$ & $0.2$ \\
Conversion efficiency & $\mu$ & $0.8$ \\
Internal energy cost & $\alpha_K$ & $0.2$ \\
Waste threshold & $Q$ & $15.0$ \\
Capital limit & $K_{\text{eff}}$ & $10.0$ \\
Control bounds & $q_{\max},I_{\max}$ & $1.0,1.0$ \\
Time horizon & $T$ & $30$ \\
\bottomrule
\end{tabular}
\end{table}

\subsection{Simulation Scenarios}
\label{sec:scenarios}

To validate the Hamilton--Jacobi (HJ) reachability-based framework, we present three simulation scenarios that test system performance under distinct waste inflow patterns. Each scenario examines the robustness of the control law and the geometric properties of the backward reachable set (BRS) under increasing levels of uncertainty and time variability. 

The simulations share the following common settings:
\begin{itemize}
    \item Time horizon: \(T = 30\) years.
    \item Viability target set:
    \[
    \mathcal{M}_Q = \{ (x,K,E)\in\mathbb{R}^3_+ : x\le 5,\; K\le 10,\; E\ge 0\}.
    \]
    \item Control vector: \(\mathbf{u}(t) = (q(t), I(t))\), where \(q(t)\) is the waste processing rate and \(I(t)\) the capital investment rate.
\end{itemize}

For all cases, the BRS is computed as the zero-sublevel set of the value function \(V(t,\mathbf{z})\) by numerically solving the constrained HJ PDE with a level-set LLF scheme (Section~\ref{sec:numerics}). The analysis emphasizes three aspects:
\begin{enumerate}
    \item The contraction or expansion of the BRS under uncertainty.
    \item The qualitative behavior of state trajectories \((x,K,E)\) under the derived optimal feedback.
    \item The sensitivity of control profiles to inflow timing and variability.
\end{enumerate}

\subsubsection{Scenario 1: Constant Inflow (Baseline Case)}
\label{scen:constant}

We first consider a nominal inflow with no variability:
\[
\eta(t) = 25,\qquad t\in[0,30].
\]
This case serves as a benchmark for validating the theoretical predictions under ideal conditions. Figure~\ref{fig:brs_eta_fix} depicts the BRS in the \((x,K)\)-plane for a fixed energy level \(E=50\). The set spans up to \(x\approx 15\), reflecting a broad margin of viability when inflow is perfectly predictable. Its smooth, convex-like geometry confirms the consistency of the numerical approximation and indicates that the optimal control law retains full flexibility in allocating resources.

\begin{figure}[!htpb]
    \centering
    \includegraphics[width=0.9\linewidth]{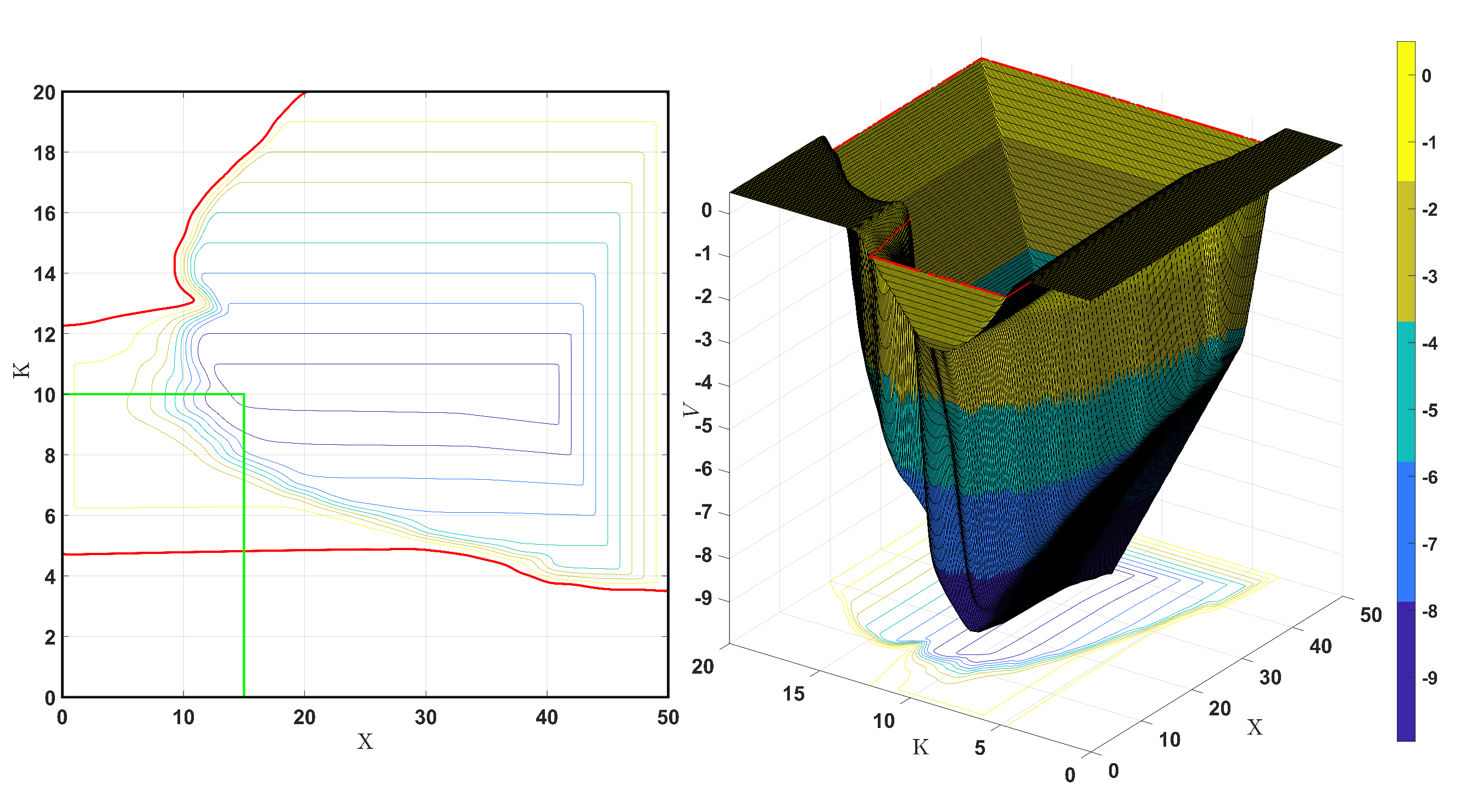}
    \caption{Backward reachable set under constant inflow, sliced at $E=50$.}
    \label{fig:brs_eta_fix}
\end{figure}

Figure~\ref{fig:states_controls_eta_fix} (left) illustrates state evolution from initial condition \((x_0,K_0,E_0)=(10,5,50)\). The waste stock decreases monotonically, capital grows steadily, and energy accumulation is persistent, leading the trajectory deep inside $\mathcal{M}_Q$. The right panel shows the corresponding $(x,K)$ path, which converges rapidly toward the target region.

\begin{figure}[!htpb]
    \centering
    \includegraphics[width=0.49\linewidth]{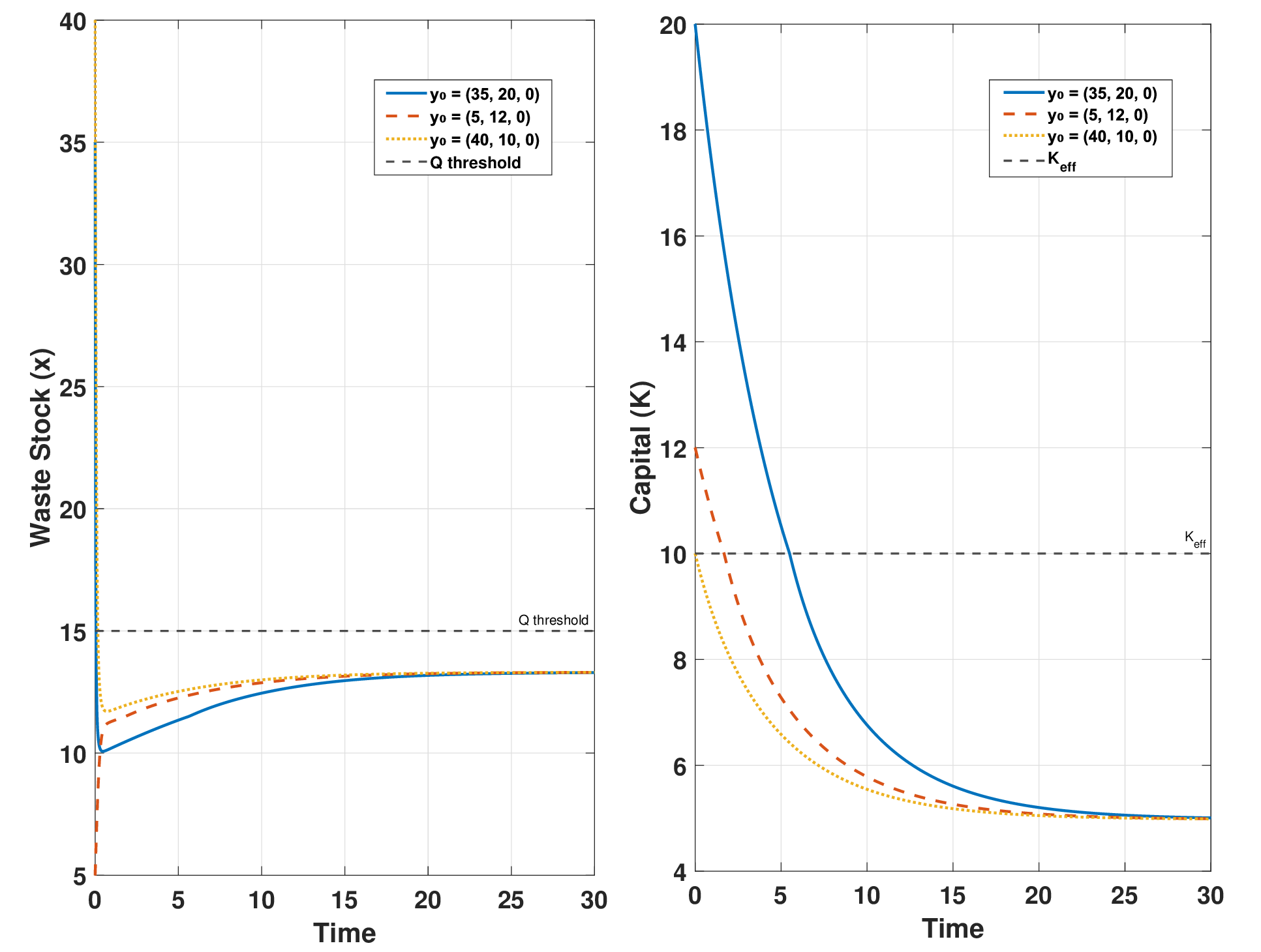}
    \includegraphics[width=0.49\linewidth]{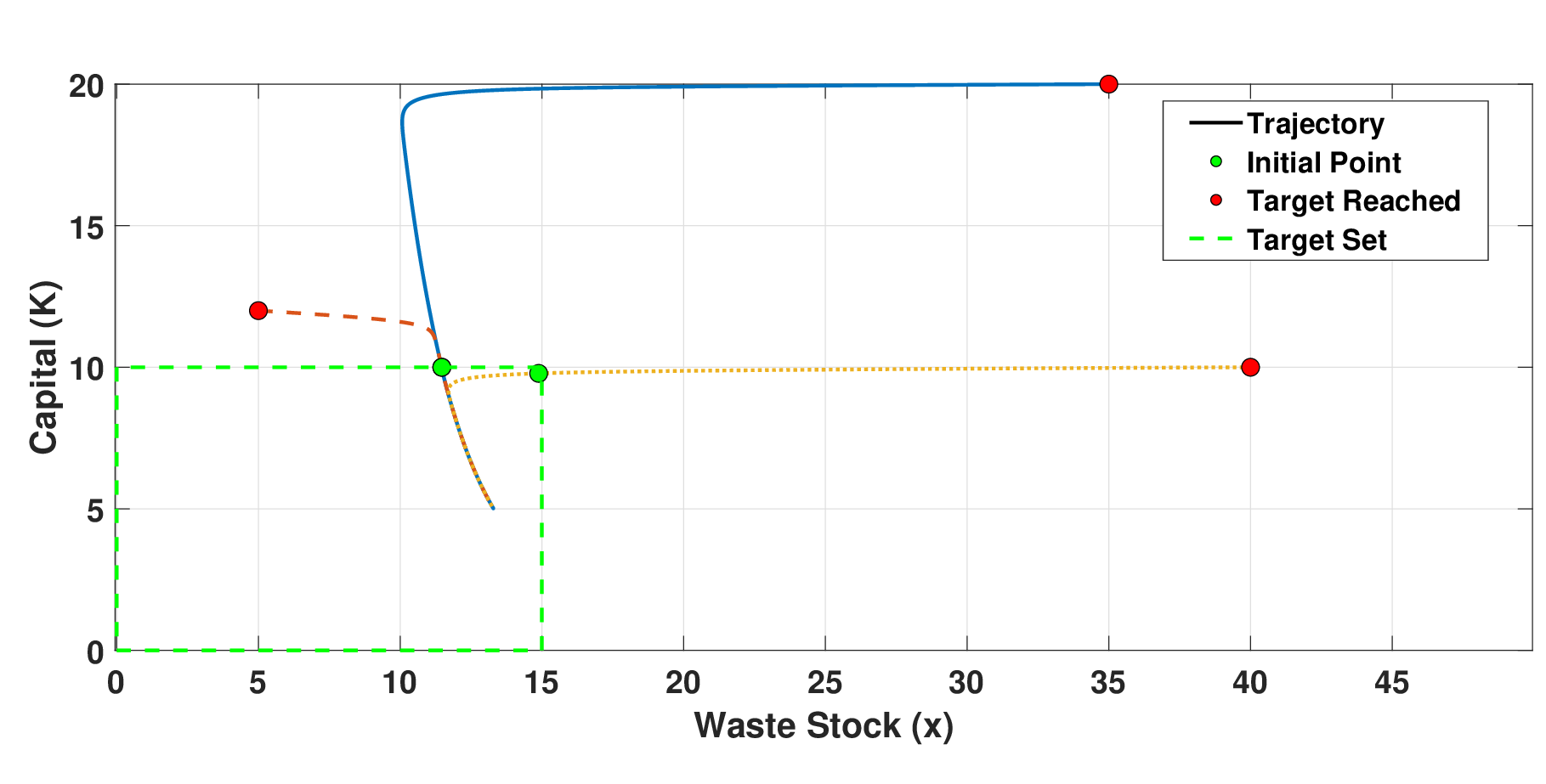}
    \caption{Left: State trajectories $(x,K,E)$ under constant inflow. Right: $(x,K)$ projection for $(10,5,50)$.}
    \label{fig:states_controls_eta_fix}
\end{figure}

Control profiles (Figure~\ref{fig:control_eta_fix}) exhibit gradual increases in $q(t)$ and $I(t)$, with no abrupt changes. This smooth behavior confirms the theoretical expectation that, in the absence of uncertainty, optimal controls exploit system dynamics without resorting to aggressive or reactive actions.

\begin{figure}[!htpb]
    \centering
    \includegraphics[width=0.7\linewidth]{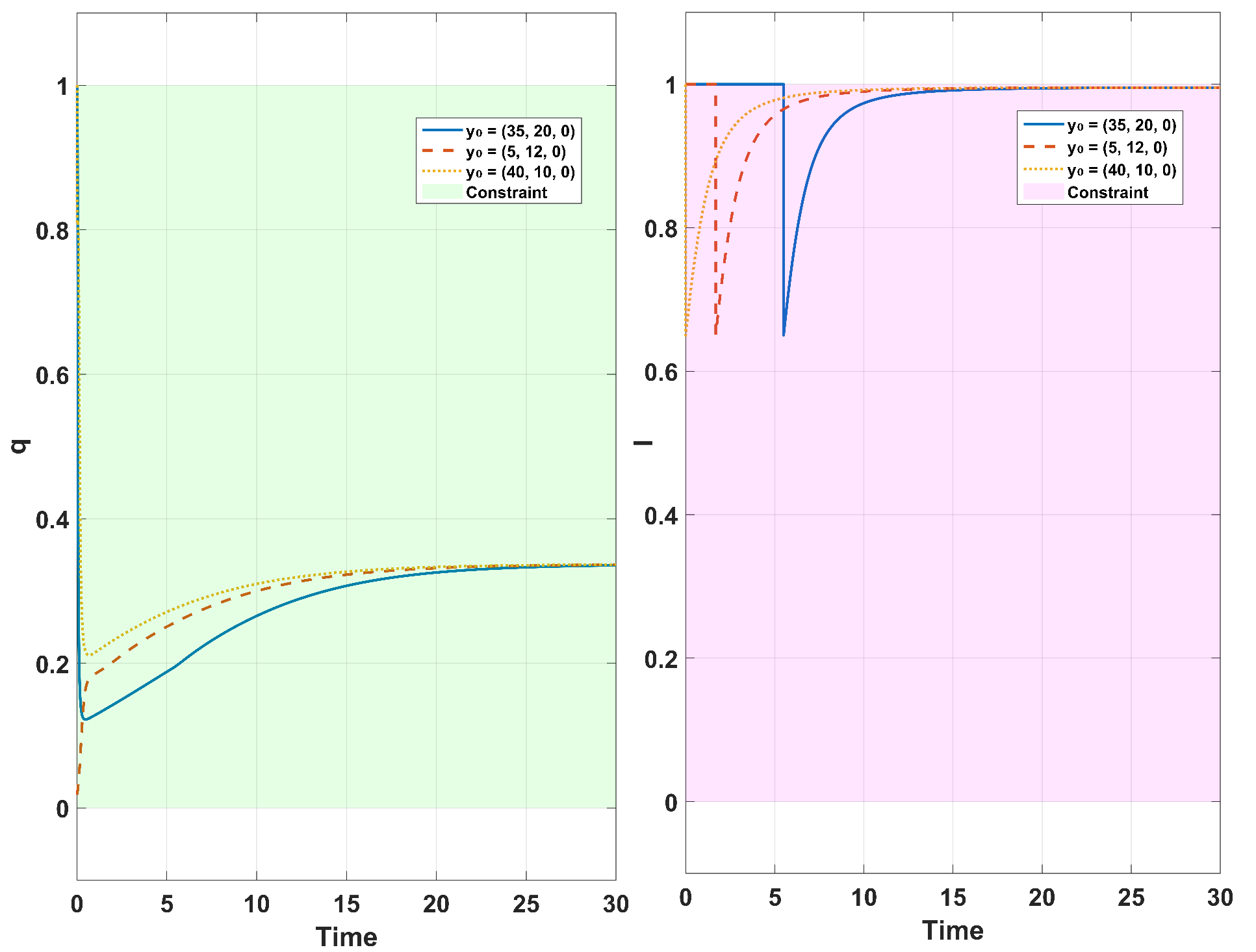}
    \caption{Control evolution under constant inflow.}
    \label{fig:control_eta_fix}
\end{figure}

Table~\ref{tab:entry_times_fix} reports entry times to $\mathcal{M}_Q$ for representative initial states. All trajectories reach viability well within the 30-year horizon, demonstrating the feasibility of the computed BRS.

\begin{table}[!htpb]
    \centering
    \caption{Entry times to $\mathcal{M}_Q$ under $\eta(t)=25$.}
    \label{tab:entry_times_fix}
    \begin{tabular}{ccc}
        \toprule
        Initial State $\mathbf{z}_0$ & $t_{\text{entry}}$ (years) & Feasible \\
        \midrule
        $(40,\ 10,\ 0)$ & 0.17 & Yes \\
        $(5,\ 12,\ 0)$  & 1.68 & Yes \\
        $(35,\ 20,\ 0)$ & 5.49 & Yes \\
        \bottomrule
    \end{tabular}
\end{table}

Scenario 1 validates the theoretical framework in an ideal setting. The large BRS volume and smooth control trajectories indicate that the HJ-based law operates efficiently when future inflow is deterministic. These results provide a baseline for contrasting with uncertainty-driven scenarios.

\vspace{0.8em}

\subsubsection{Scenario 2: Stepwise Inflow with $\pm 10\%$ Uncertainty}
\label{scen:stepwise}

We now introduce bounded uncertainty:
\[
\eta(t)\in[22.5,27.5].
\]
Three piecewise-constant profiles are defined over three equal intervals ($10$ years each):
\[
\eta_1(t)=\{27.5,25.0,22.5\},\;
\eta_2(t)=\{25.0,27.5,22.5\},\;
\eta_3(t)=\{22.5,27.5,25.0\}.
\]

Figure~\ref{fig:brs_eta10} displays the BRS for worst-case disturbances. Compared to Scenario 1, the set contracts significantly along the $x$-dimension, reducing tolerance for large initial waste loads. This geometric shrinkage reflects the system's conservative strategy under inflow ambiguity, consistent with viability theory: fewer states guarantee safety under adversarial conditions.

\begin{figure}[!htpb]
    \centering
    \includegraphics[width=0.9\linewidth]{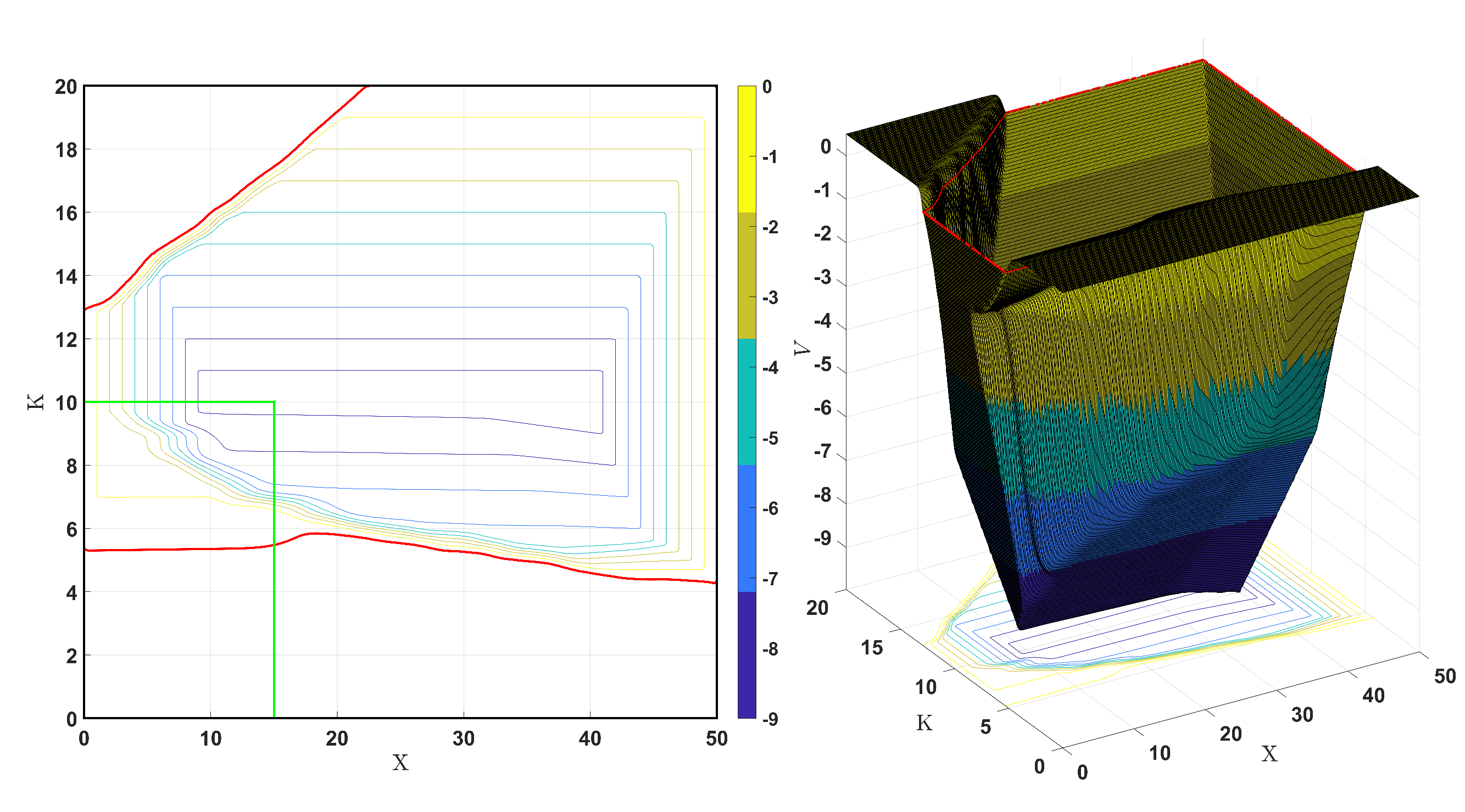}
    \caption{Backward reachable set under $\pm 10\%$ inflow uncertainty.}
    \label{fig:brs_eta10}
\end{figure}

Figure~\ref{fig:state_control_eta10} (left) shows state trajectories for the three inflow profiles from initial state $(40,12,0)$. The early-peak inflow accelerates energy growth, while the delayed-peak case slows convergence. The control responses (right) adapt accordingly: aggressive $q(t)$ and $I(t)$ for early peaks, moderated efforts for low-start profiles.

\begin{figure}[!htpb]
    \centering
    \includegraphics[width=0.49\linewidth]{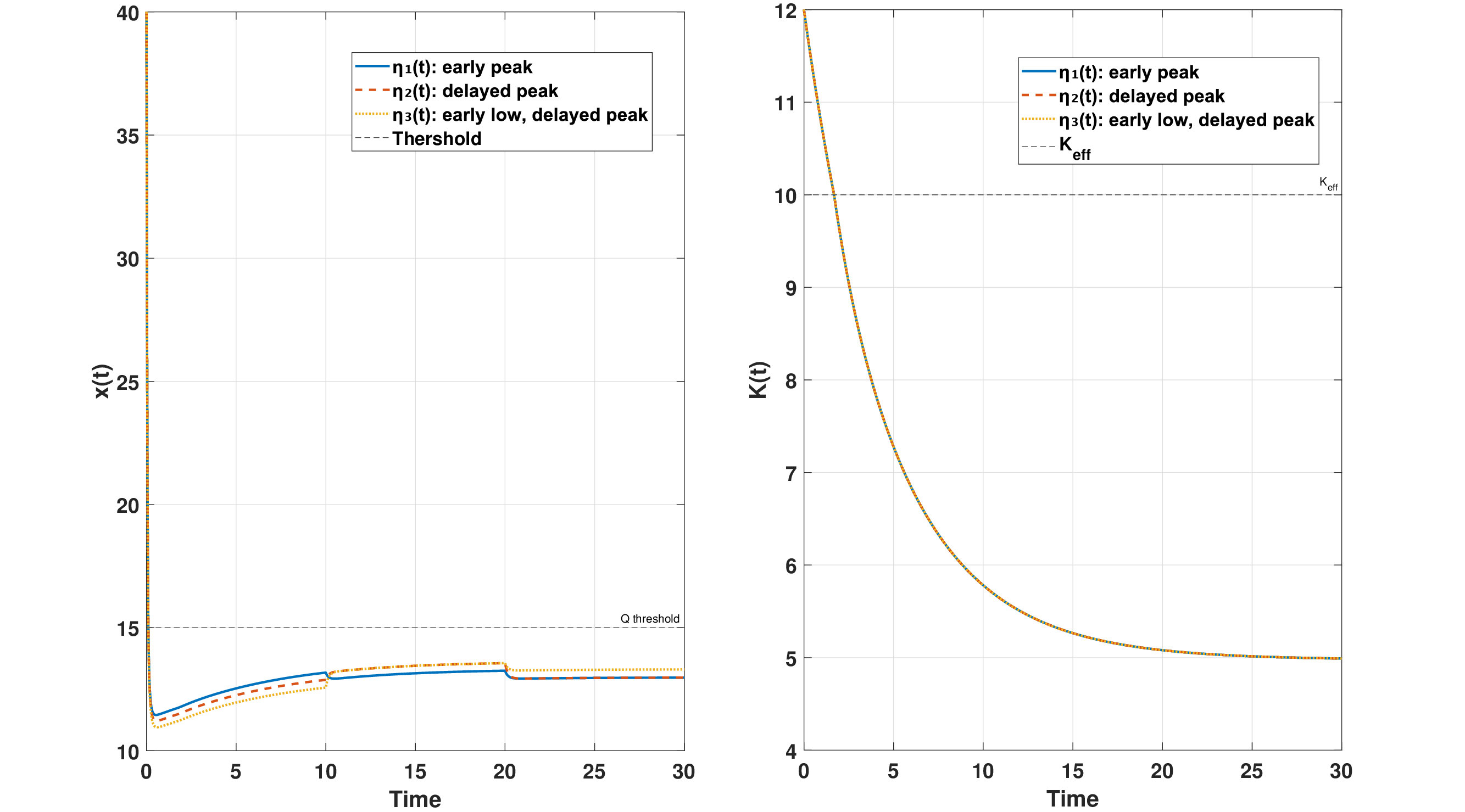}
    \includegraphics[width=0.49\linewidth]{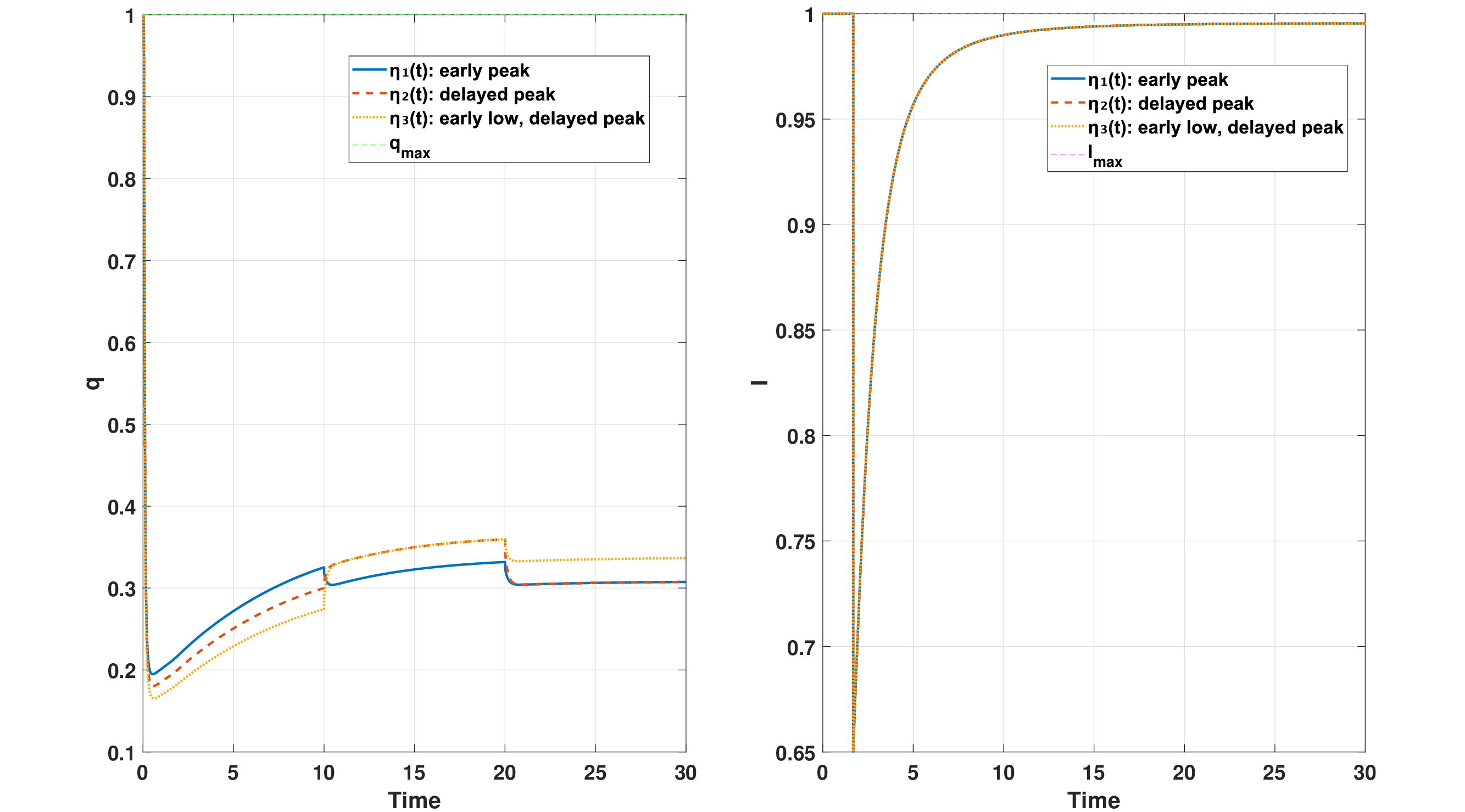}
    \caption{Left: State trajectories under three inflow profiles. Right: Control strategies.}
    \label{fig:state_control_eta10}
\end{figure}

Energy trajectories (Figure~\ref{fig:energy_eta10}) confirm this pattern: front-loaded inflows yield higher early energy gains, while back-loaded scenarios delay production. Yet, all cases achieve viability, highlighting robustness despite performance variability.

\begin{figure}[!htpb]
    \centering
    \includegraphics[width=0.7\linewidth]{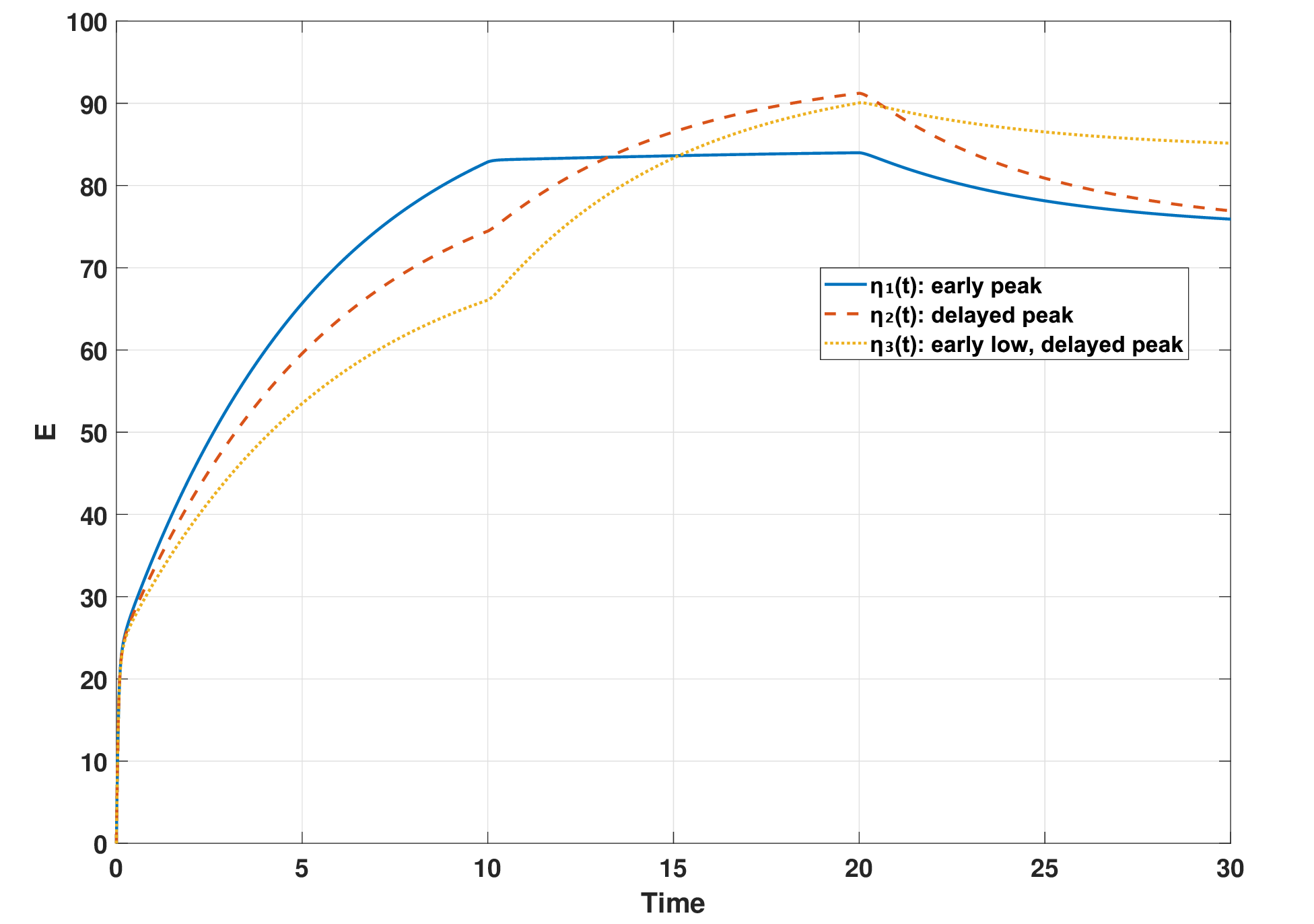}
    \caption{Energy evolution $E(t)$ under stepwise inflow variations.}
    \label{fig:energy_eta10}
\end{figure}

Scenario 2 demonstrates the system's resilience under bounded uncertainty but reveals critical structural insights:  
(i) BRS contraction implies increased sensitivity of initial feasibility to uncertainty;  
(ii) The timing of inflow peaks significantly influences energy efficiency, even if cumulative inflow remains constant;  
(iii) Control effort reallocates dynamically, preserving viability at the expense of delayed or uneven energy gains.

\vspace{0.8em}

\subsubsection{Scenario 3: Periodic Inflow with Jump Disturbances}
\label{scen:periodic}

Finally, we consider a dynamic inflow combining smooth seasonal variability with discrete shocks:
\[
\eta(t)=\underbrace{\eta_\omega\sin^2\left(\frac{\pi t}{\rho_\omega}\right)+\tau_\omega}_{\text{baseline oscillation}}+\underbrace{J(t)}_{\text{jumps}},
\]
where $\eta_\omega=20$, $\rho_\omega=5$, $\tau_\omega=15$, and:
\[
J(t)=\begin{cases}
+5,&10\le t<20,\\
-3,&20\le t<30,\\
0,&\text{otherwise}.
\end{cases}
\]

This hybrid pattern captures realistic fluctuations from seasonal waste generation and abrupt shocks (e.g., policy changes, demand surges). The initial condition is $(x_0,K_0,E_0)=(5,12,0)$.

Figure~\ref{figrechability_sin} shows the BRS under these conditions. Compared to the baseline, the set is markedly smaller near high $x$ values, confirming that combined variability and shocks impose stricter viability constraints.

\begin{figure}[!htpb]
    \centering
    \includegraphics[width=0.9\linewidth]{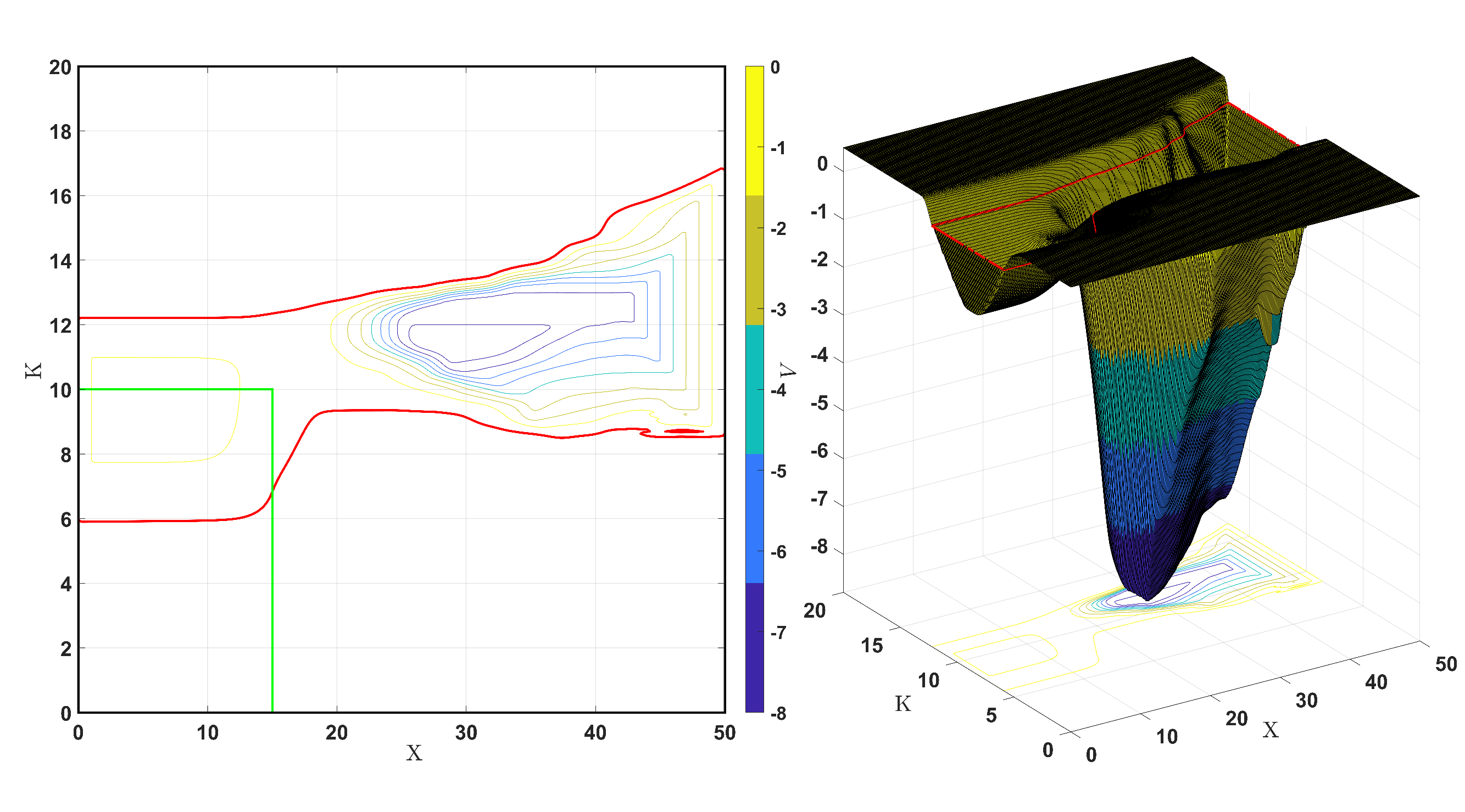}
    \caption{BRS under periodic inflow with jump disturbances.}
    \label{figrechability_sin}
\end{figure}

State trajectories (Figure~\ref{fig:state_control_sinus}, left) exhibit oscillatory patterns synchronized with inflow cycles, while control actions (right) adapt smoothly despite abrupt inflow jumps, avoiding destabilizing behavior.

\begin{figure}[!htpb]
    \centering
    \includegraphics[width=0.49\linewidth]{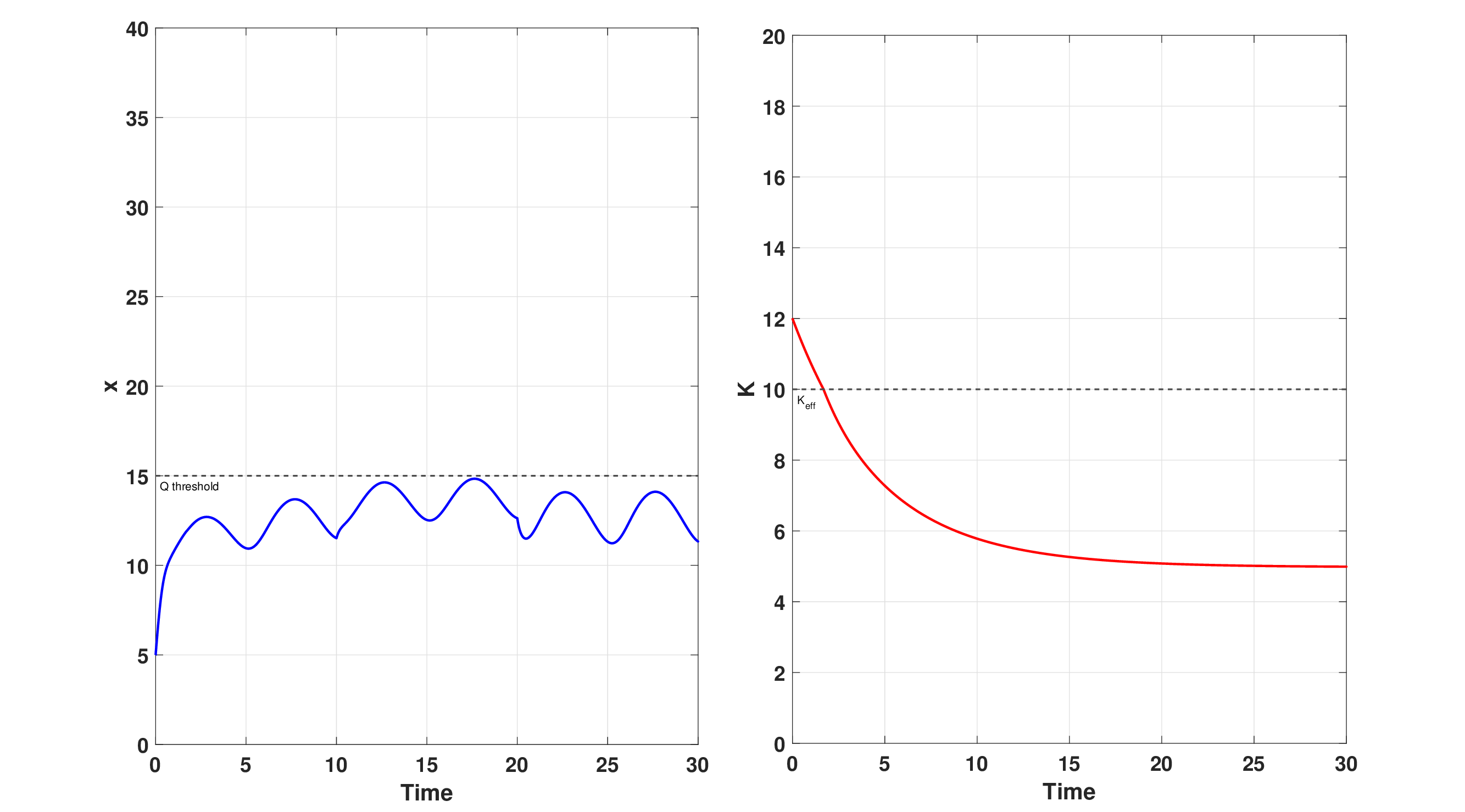}
    \includegraphics[width=0.49\linewidth]{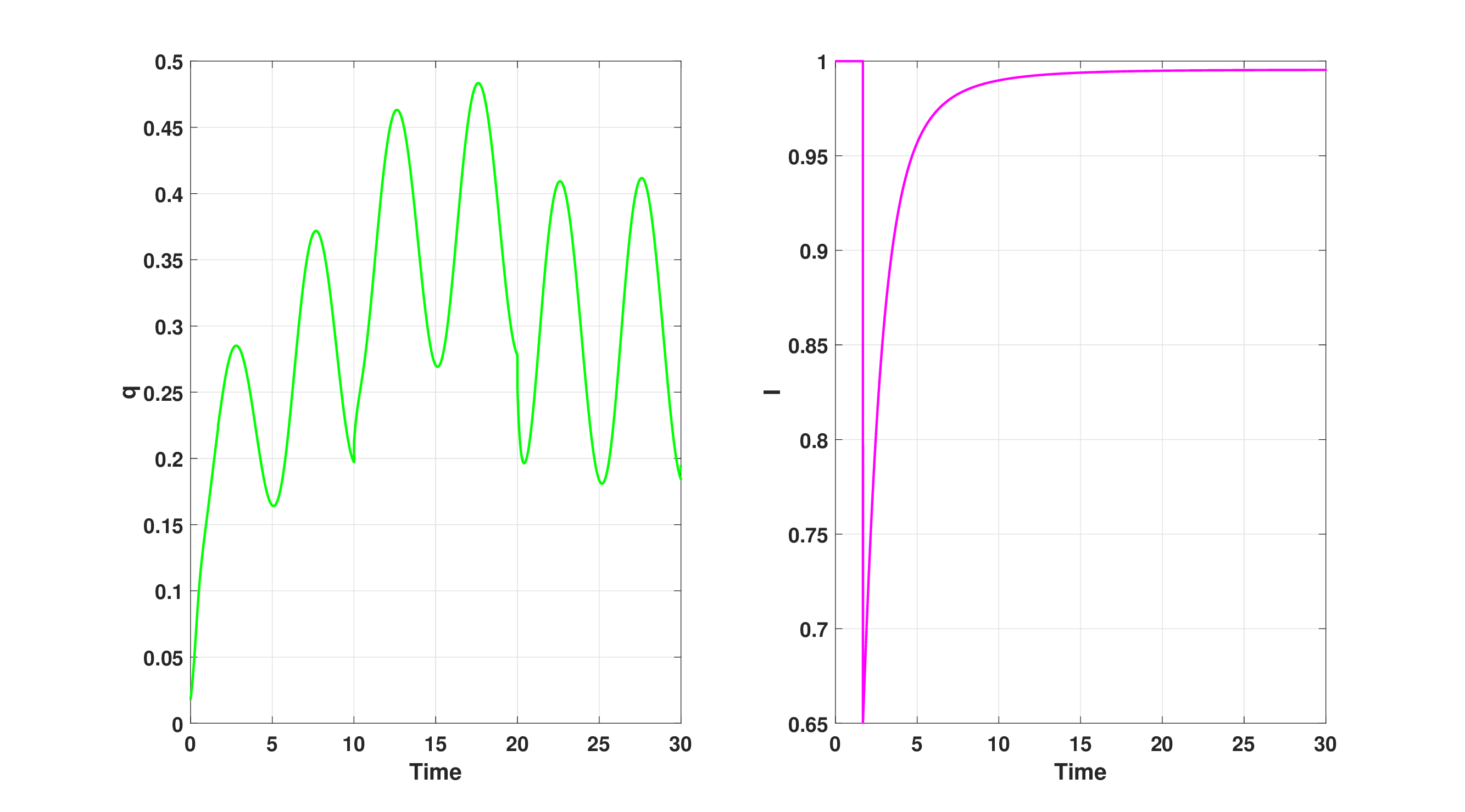}
    \caption{Left: State evolution under sinusoidal inflow. Right: Control response.}
    \label{fig:state_control_sinus}
\end{figure}

Figure~\ref{fig:energy_sinus} confirms that energy output remains monotonic despite nonstationary inflow, showcasing the robustness of the control synthesis.

\begin{figure}[!htpb]
    \centering
    \includegraphics[width=0.7\linewidth]{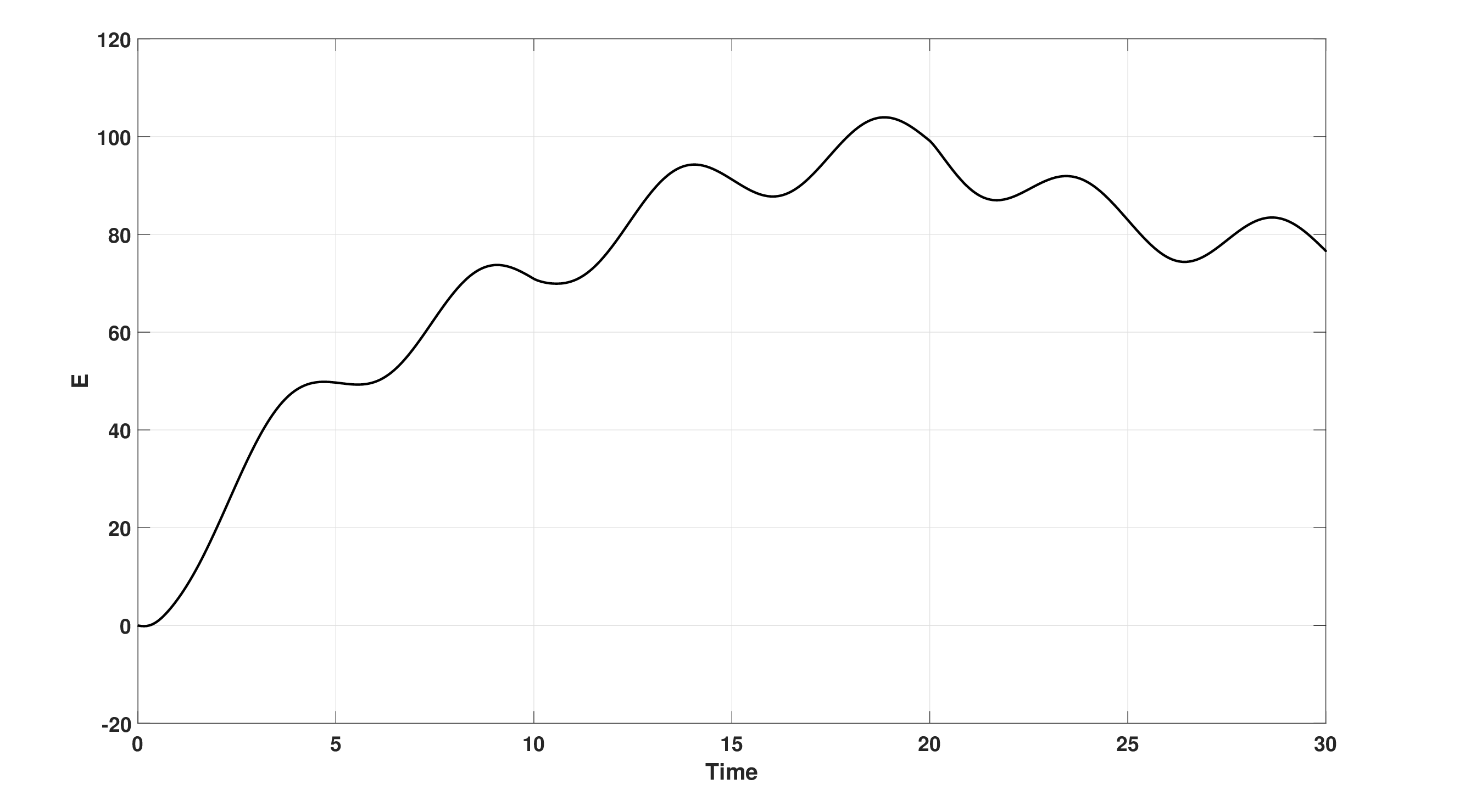}
    \caption{Energy accumulation under periodic inflow with shocks.}
    \label{fig:energy_sinus}
\end{figure}

Scenario 3 underscores the scalability of the HJ framework to highly nonstationary environments. Despite severe inflow irregularities, the controller ensures state constraint compliance and steady energy growth, albeit with reduced feasibility margins (BRS shrinkage). This highlights a critical insight: while viability can be preserved, performance guarantees (e.g., energy efficiency) degrade gracefully rather than catastrophically—a key advantage of adopting a reachability-based control approach.

\vspace{0.8em}

\paragraph{Summary of Scenario Analysis.}
Across all scenarios, the following structural trends emerge:
\begin{itemize}
    \item \textbf{Uncertainty reduces feasibility:} BRS contraction correlates with inflow variability, consistent with robust reachability theory.
    \item \textbf{Inflow timing matters:} For equal cumulative inflow, early availability accelerates energy production and lowers control stress.
    \item \textbf{Controller robustness:} Optimal controls remain well-behaved even under discontinuous inflow patterns, avoiding reactive instability.
\end{itemize}
These observations validate the HJ-based framework as a mathematically sound and practically robust tool for long-term planning in Waste-to-Energy systems.

\section{Conclusion}

This paper presented a mathematically rigorous framework for robust reachability analysis in Waste-to-Energy (WtE) systems operating under inflow uncertainty. The problem was cast as a zero-sum differential game between the control policy and an adversarial disturbance, leading to a Hamilton--Jacobi (HJ) formulation. The backward reachable set (BRS), which characterizes all states guaranteeing long-term viability, was expressed as the zero sublevel set of a robust value function. We established that this value function is the unique viscosity solution of a constrained Hamilton--Jacobi--Bellman (HJB) equation, providing formal safety guarantees under worst-case uncertainty.

Numerical experiments using a monotone level-set solver based on the Local Lax--Friedrichs scheme confirmed the convergence of the proposed method and its effectiveness across diverse inflow scenarios. Results highlighted a key structural insight: increasing uncertainty induces a systematic contraction of the BRS, thereby reducing operational flexibility. Nonetheless, the derived control strategies successfully enforced viability constraints and ensured convergence to the target set in all tested cases, demonstrating both robustness and practical applicability.

Beyond WtE applications, this work establishes a general viability-based paradigm for safety-critical and resource-constrained systems where uncertainty management is essential. Future research will address three directions: (i) incorporating stochastic and hybrid uncertainty models for more realistic representations, (ii) developing scalable decomposition algorithms to mitigate the curse of dimensionality, and (iii) integrating multi-objective criteria—such as economic performance—into the HJ reachability framework. These extensions aim to bridge theoretical rigor with real-world decision-making, advancing the role of Hamilton--Jacobi methods in sustainable energy and environmental system design.

\end{document}